\documentclass[11pt,leqno,amscd,pstricks,verbatim]{amsart}
\usepackage{amsmath,amsthm, amssymb, latexsym,}
\usepackage{cases}
\usepackage{mathrsfs}
\usepackage{pifont}
\usepackage{amsfonts}
\usepackage{stmaryrd}
\usepackage{color}

\usepackage[english]{babel}
\usepackage[latin1]{inputenc}
\usepackage[numbers,sort&compress]{natbib}
\usepackage[all]{xy}

\usepackage{hyperref}
\usepackage[T1]{fontenc}

\usepackage{ dsfont}

\usepackage{graphicx,xcolor}

\newtheorem{theorem}{Theorem}[section]
\newtheorem{lemma}[theorem]{Lemma}
\newtheorem{corollary}[theorem]{Corollary}

\theoremstyle{definition}

\theoremstyle{remark}
\newtheorem{remark}[theorem]{Remark}

\numberwithin{equation}{section}

\newcommand{\on}[1]{\operatorname{#1}}

\newcommand{\un}[1]{\underline{#1}}

\newcommand{\mf}[1]{\mathfrak{#1}} 
\newcommand{\p}{\partial} 

\newcommand{\ds}{\oplus} 
\newcommand{\bds}{\bigoplus} 

\newcommand{\inj}{\hookrightarrow} 

\def\bbf{\mathbb{F}}
\def\ad{\mbox{\rm ad}\,}
\def\ggg{\mathfrak{g}}
\def\al{\alpha}
\def\fraka{\mathfrak{A}}
\def\veps{\epsilon}
\def\div{\text{\rm div}}

\begin{document}

\title[On structure of graded restricted simple Lie algebras of Cartan type]
{On structure of graded restricted simple Lie algebras of Cartan type as modules over the Witt algebra}



\thanks{This work is partially supported by the National Natural Science Foundation of China (Grant Nos. 11771279 and 12071136), the Fundamental Research Funds of Yunnan Province (Grant No. 2020J0375) and the Fundamental Research Funds of YNUFE (Grant No. 80059900196).}

\author{Ke Ou}
\address{School of Statistics and Mathematics, Yunnan University of Finance and Economics,
Kunming, 650221, China.}\email{keou@ynufe.edu.cn}

\author{Yu-Feng Yao}
\address{Department of Mathematics, Shanghai Maritime University, Shanghai, 201306, China.}\email{yfyao@shmtu.edu.cn}

\begin{abstract}
Any graded restricted simple Lie algebra of Cartan type contains a subalgebra isomorphic to the Witt algebra over a field of prime characteristic. As some analogue of study on branching rules for restricted non-classical Lie algebras,   it is shown that each graded restricted simple Lie algebra of Cartan type can be decomposed into a direct sum of restricted baby Verma modules and simple modules as an adjoint module over the Witt algebra. In particular, the composition factors are precisely determined.
\end{abstract}

\maketitle
\section{Introduction}
It is well known that in the late 1930's E. Witt firstly found a
non-classical simple Lie algebra over prime characteristic field
which is called the Witt algebra $W(1)$. This contributed to the study
of non-classical simple Lie algebras which
were called Cartan-type Lie algebras later. A well-known classification result on simple modular Lie algebras asserts that each finite dimensional (restricted) simple Lie algebra over a field of prime characteristic  $p>5$ is of either classical type or Cartan type (cf.
\cite{BW,PS}). There are four families of simple Lie algebras $X(n)$ of Cartan type $X$ for $X\in\{W,S,H,K\}$. They are subalgebras of derivation algebra of truncated polynomial algebras. Each simple Lie algebra of Cartan type contains a subalgebra isomorphic to the Witt algebra, which plays a similar role as the three dimensional simple Lie algebra $\mathfrak{sl}_2$ of type $A_1$ in classical simple Lie algebras. In view of this point,
the Witt algebra is a "fundamental" non-classical simple Lie algebra.

The representation theory of the Witt algebra $W(1)$ was firstly
studied by Ho-Jui Chang in the early 1940's (cf. \cite{Ch}). Its irreducible representations were completely determined. Based on this result, in the present paper we precisely determine the structure of graded restricted simple Lie algebras of Cartan type as adjoint modules over the Witt algebra. It is shown that any graded restricted simple Lie algebra of Cartan type can be decomposed as a direct sum of restricted baby Verma modules and simple modules over the Witt algebra. As a consequence, the composition factors are precisely determined. We hope the study on the decomposition of $ X(n) $ as a $ W(1) $-modules will provide some useful and interesting intrinsic observation on the structure of irreducible restricted $ X(n) $-modules and  branching rules in resticted representation category for $X(n)$, where $X\in \{W,S,H,K\} $.

This paper is organized as follows. In section 2, we introduce some basic concepts on restricted Lie algebras and their (restricted) representations, and the algebra structure on graded Lie algebras of Cartan type. In particular, we present precisely the embedding of the Witt algebra to the four families of Lie algebras of Cartan type. Moreover, the restricted representation theory of the Witt algebra is recalled. Section 3 is devoted to studying the decomposition of the Jacobson-Witt algebra as a module over the Witt algebra into a direct sum of submodules. In section 4, we first precisely give a basis for the special algebra. By using this basis, we decompose the special algebra into a direct sum of restricted baby Verma modules and simple modules over the Witt algebra. Sections 5 and 6 are devoted to determining the decomposition of the Hamiltonian algebra and the contact algebra as direct sums of restricted baby Verma modules and simple modules over the Witt algebra, respectively.

\section{Preliminaries}
Throughout this paper, $\mathbb{F}$ is assumed to be an
algebraically closed field of prime characteristic $p>2$. All
modules (vector spaces) are over $\mathbb{F}$ and finite-dimensional. Set $I=\{0,1,\cdots,p-1 \}. $ For a finite set $S$, let $|S|$ denote the number of elements in $S$. For a Lie algebra $\ggg$, let $U(\ggg)$ be its universal enveloping algebra, and $Z(\ggg)$ be the center of $U(\ggg)$. For a $\ggg$-module $M$, let $[M]$ be the formal sum of all composition factors of $M$.
\subsection{Restricted Lie algebras and their irreducible representations}
Recall that a restricted Lie algebra $\ggg$ over $\bbf$ is a Lie algebra with a so-called
restricted mapping $[p]: \ggg\rightarrow \ggg$ sending $x\mapsto
x^{[p]}$  satisfying that $ \ad (x^{[p]})=(\ad x)^p$ and that
$\xi:\ggg\rightarrow Z(\ggg)$ sending $x\mapsto x^p-x^{[p]}$ is semi-linear.

For a restricted Lie algebra $(\ggg, [p])$ and a simple $\ggg$-module $M$, since $x^p-x^{[p]}\in Z(\ggg)$ for any $x\in\ggg$, the element $x^p-x^{[p]}$ must act as a scalar, denoted by $\chi(x)^p$. The semilinearity of $\xi$ implies that $\chi\in\ggg^*$. In general, a $\ggg$-module $M$ is said to be
$\chi$-reduced if $x^p\cdot v-x^{[p]}\cdot v=\chi(x)^pv$ for all
$x\in\ggg, v\in M$. In particular, it is called a
restricted module if $\chi=0$. Let $U_{\chi}(\ggg)$ be the quotient of the universal
enveloping algebra $U(\ggg)$ by the
ideal generated by $\{x^p-x^{[p]}-\chi(x)^p\mid
x\in\ggg\}$ which is called a $\chi$-reduced enveloping algebra of $\ggg$, i.e.,
$U_{\chi}(\ggg)=U(\ggg)/(x^p-x^{[p]}-\chi(x)^p\mid
x\in\ggg_{\bar{0}})$. If $\chi=0$, the algebra $U_0(\ggg)$ is
called the restricted enveloping algebra and denoted by
$u(\ggg)$ for brevity. All the $\chi$-reduced (resp. restricted)
$\ggg$-modules constitute a full subcategory of the $\ggg$-module
category, which coincides with the $U_{\chi}(\ggg)$ (resp.
$u(\ggg)$)-module category. Each simple $\ggg$-module is a
$U_{\chi}(\ggg)$-module for a unique $\chi\in\ggg^*$.

\subsection{Graded Lie algebras of Cartan type}\label{section Cartan type lie algebras}

Fix a positive integer $n$. Denote by $A(n)$ the index set
$\{\al=(\al_1,\cdots,\al_n)\mid 0\leq \al_i\leq p-1,
i=1,2,\cdots,n\}$, and denote
$(p-1,\cdots,p-1)$ by $\tau$ for brevity. We have a truncated polynomial algebra
$\fraka(n)$ which is by definition a commutative associative
algebra with a basis $\{x^\al\mid \al\in A(n)\}$, and
multiplication subject to the following rule
\begin{equation*}\label{multiplication rule}
x^\al x^\beta=x^{\alpha+\beta},\quad\forall\,\al,\beta\in A(n),
\end{equation*}
additionally with
$$x^\alpha=0 {\text{ if }} \alpha\notin A(n);\;\;\;
x_i:=x^{\veps_i} {\text{ for }}
\veps_i=(\delta_{1,\,i},\cdots,\delta_{n,\,i}).$$
There is a natural graded structure on $\fraka(n)$, and
consequently a filtered structure there. The gradation and
filtration of $\fraka(n)$ induce the corresponding ones on
the so-called Jacobson-Witt algebra $W(n)$, which is the derivation algebra of $\fraka(n)$. Then $W(n)$ is free $\fraka(n)$-module with a basis $\{\partial_1,\cdots, \partial_n\}$, where $\partial_i(x_j)=\delta_{ij}, 1\leq i, j\leq n$. For more details, the readers are referred to the reference
\cite{SF, St}.

We can get other three series of subalgebras in $W(n)$,
which are called (graded) Cartan type Lie algebras of series $S,H$,
and $K$  respectively, arising from the three exterior differentials
$\omega_S,\omega_H,\omega_K$. Below, we recall the definitions, and
cite some basic notations and facts we need later. The definitions
here will be given by using some operators (cf. \cite[Chapter
4]{St}), instead of using the original differential forms (cf.
\cite{KS}).

Set $\widetilde{S}(n)=\{D\in W(n)\mid \div
(D)=0\}$, where $\div (\sum f_i\partial_i)=\sum \partial_i(f_i)$ for
any $\sum f_i\partial_i\in W(n)$. Then by definition, the derived
algebra of $\widetilde{S}(n)$ is called the special algebra
$S(n)$, i.e. $S(n)=[\widetilde{S}(n),
\widetilde{S}(n)]$.

The Hamiltonian algebra is by definition
$H(2r)=\bbf$-span$\{D_{H}(x^{\al})\mid 0\prec\al\prec\tau\}$. Here
$D_H$ is the Hamiltonian operator from $\fraka(2r)$ to
$W(2r)$ defined as follows:
$$
\aligned D_{H}:\quad\quad \fraka(2r) &\longrightarrow
W(2r)\cr f &\longmapsto
D_{H}(f)=\sum\limits_{i=1}^{2r}\sigma(i)\partial_i(f)\partial_{i^{\prime}}
\endaligned
$$
where
\[
\sigma(i):=\left\{
\begin{array}{ll}
1,& \text{if}\quad 1\leq i\leq r, \\
-1,& \text{if}\quad r+1\leq i\leq 2r,
\end{array}
\right.
\]
and
\[
i^{\prime}:=\left\{
\begin{array}{ll}
i+r,& \text{if}\quad 1\leq i\leq r, \\
i-r,& \text{if}\quad r+1\leq i\leq 2r.
\end{array}
\right.
\]

Set $\widetilde{K}(2r+1)=\bbf$-span$\{D_K(x^{\al})\mid
\al\in A(2r+1)\}$, where the contact operator $D_K$ from
$\fraka(2r+1)$ to $W(2r+1)$ is defined as follows:
$$
\aligned D_{K}:\quad\quad \fraka(2r+1) &\longrightarrow
W(2r+1)\cr f &\longmapsto
D_{K}(f)=\sum\limits_{i=1}^{2r+1}f_i\partial_i
\endaligned
$$
where
$$ f_j=x_j\partial_{2r+1}(f)+\sigma(j^{\prime})\partial_{j^{\prime}}(f),\,\,j\leq 2r,$$
$$f_{2r+1}=2f-\sum\limits_{i=1}^{2r}\sigma(j)x_jf_{j^{\prime}}.$$
The contact algebra $K(2r+1)$ is the derived algebra of
$\widetilde{K}(2r+1)$, i.e.
$K(2r+1)=[\widetilde{K}(2r+1), \widetilde{K}(2r+1)]$.

Both subalgebras $S(n)$ and $H(n)$ naturally inherit the
graded and filtered structure of $W(n)$. But the contact algebra
$K(2r+1)$ is not a graded subalgebra of $W(2r+1)$. One can
define a new gradation on $K(2r+1)$ which is not inherited from
the gradation of $W(2r+1)$. For that, define
$||\al||=\sum_{i=1}^{2r+1}\al_i+\al_{2r+1}-2$ for $\al\in
A(2r+1)$ and $K(2r+1)_{[i]}=\bbf$-span$\{D_K(x^{\al})\mid
||\al||=i\}$. Then $K(2r+1)=\bigoplus_{i\geq
-2}K(2r+1)_{[i]}$ is a gradation of $K(2r+1)$. Associated
with this gradation, one can also obtain the corresponding
filtration
$$K(2r+1)=K(2r+1)_{-2}\supset
K(2r+1)_{-1}\supset\cdots$$ where
$K(2r+1)_i=\bigoplus_{j\geq i}K(2r+1)_{[j]}.$

Let $L=X(n)$, $X\in\{W,S,H,K\}$. Then $L$ is a restricted Lie algebra with the restricted $[p]$-mapping given by taking the $p$-th power of derivations. Moreover, $L$ has a $\mathbb{Z}$-grading $L=\oplus_{i\geq -1-\delta_{XK}}L_{[i]}$, associated with which there is a natural filtration $L=L_{-1-\delta_{XK}}\supset L_{0-\delta_{XK}}\supset\cdots$ with $L_{j}=\oplus_{i\geq j}L_{[i]}$ for $j\geq -1-\delta_{XK}$.

\subsection{Embeddings of the Witt algebra to restricted Lie algebras of Cartan type} Define the linear mappings from the Witt algebra to restricted Lie algebras of Cartan type as follows.
$$
\aligned \Theta_{W}:\quad\quad W(1) &\longrightarrow
W(n)\cr x_1^i\partial_1 &\longmapsto x_1^i\partial_1,\quad \forall\,0\leq i\leq p-1,\cr\\
\Theta_{S}:\quad\quad W(1) &\longrightarrow
S(n)\cr x_1^i\partial_1 &\longmapsto D_{12}(x_1^ix_2)=x_1^i\partial_1-ix_1^{i-1}x_2\partial_2,\quad \forall\,0\leq i\leq p-1,\cr\\
\Theta_{H}:\quad\quad W(1) &\longrightarrow
H(2r)\cr x_1^i\partial_1 &\longmapsto -D_H(x_1^ix_{r+1})=x_1^i\partial_1-ix_1^{i-1}x_{r+1}\partial_{r+1},\quad \forall\,0\leq i\leq p-1,\cr\\
\Theta_{K}:\quad\quad W(1) &\longrightarrow
K(2r+1)\cr x_1^i\partial_1 &\longmapsto \frac{1}{2}D_K(x_{2r+1}^i),\quad \forall\,0\leq i\leq p-1.
\endaligned
$$
The following result asserts that the Witt algebra is a restricted subalgebra of any restricted Lie algebra of Cartan type. The proof follows from straightforward computation, we omit the details.
\begin{lemma}\label{embedding}
Keep notations as above. Then $\Theta_{X}$ is an injective restricted Lie algebra homomorphism from the Witt algebra $W(1)$ to the restricted Lie algebra $X(n)$ of Cartan type $X$ for $X\in\{W,S,H,K\}$.
\end{lemma}

\begin{remark}
Thanks to Lemma \ref{embedding}, any Lie algebra of Cartan type is a restricted module over the Witt algebra under the adjoint action. We will determine this module structure in the sequel sections.
\end{remark}

\subsection{Restricted representations of the Witt algebra $W(1)$}\label{witt rep} In this subsection, we always assume that $\ggg=W(1)$ is the Witt algebra over $\bbf$. We will recall the known classification results on simple $\ggg$-modules given by Chang in \cite{Ch}. Recall that $\ggg$ has a natural $\mathbb{Z}$-gradation $\ggg=\oplus_{i=-1}^{p-2}\ggg_{[i]}$, associated with which there is a filtration $\ggg=\ggg_{-1}\supset\ggg_0\cdots\supset\ggg_{p-2}\supset 0$, where $\ggg_i=\oplus_{j\geq i}\,\ggg_{[j]}$ for $-1\leq i\leq p-2$. For any $\lambda\in I$, let $\bbf_{\lambda}$ be the one dimensional restricted $\ggg_{[0]}$-module given by multiplication by the scalar $\lambda$. Then we can regard $\bbf_{\lambda}$ as a restricted $\ggg_0$-module with trivial action by $\ggg_1$. Let $V(\lambda)=u(\ggg)\otimes_{u(\ggg_0)}\bbf_{\lambda}$ which is called a restricted baby Verma $\ggg$-module. Each restricted baby Verma $\ggg$-module $V(\lambda)$ has a unique simple quotient denoted by $L(\lambda)$ for $\lambda\in I$. Then the set $\{L(\lambda)\mid \lambda\in I\}$ exhausts all non-isomorphic irreducible restricted $\ggg$-modules. Moreover, $L(\lambda)=V(\lambda)$ if and only if $0<\lambda<p-1$. While both $V(0)$ and $V(p-1)$ have two composition factors $L(0)$ and $L(p-1)$. The natural module $A(1)$ is isomorphic to $V(p-1)$, while the adjoint module $W(1)$ is isomorphic to $L(p-2)$.

%
%

\section{Structure of the Jacobson-Witt algebras as modules over the Witt algebra}
In this section, we study the structure of the Jacobson-Witt algebra $W(n)$ as a module over the Witt algebra $W(1)$. For that, denote $ x^{\un{i}}=x_2^{i_2}\cdots x_n^{i_n} $ for any $ \un{i}=(i_2,\cdots,i_n)\in I^{n-1}$. Then
\begin{equation}\label{w decom}
 W(n)=\bigoplus_{j=1}^n \bigoplus_{\un{i}\in I^{n-1}} A(1)x^{\un{i}}\p_j.
\end{equation}
Moreover, each summand in (\ref{w decom}) is a $W(1)$-module. More precisely, for each $ \un{i}\in I^{n-1} $, we have the following isomorphism as modules over the Witt algebra $W(1)$,
$$A(1)x^{\un{i}}\p_j\cong\begin{cases}
W(1), &\text{ if } j=1,\cr
A(1), & \text{ if } 2\leq j\leq n.
\end{cases}
$$
Consequently, we have
\begin{theorem}
As a module over the Witt algebra $W(1)$, we have
\begin{equation}\label{w-1}
W(n)\cong V(p-1)^{\oplus(n-1)p^{n-1}}\oplus L(p-2)^{\oplus p^{n-1}}.
\end{equation}
Hence,
\begin{equation}\label{w-2}
[W(n)]=(n-1)p^{n-1}\left([L(0)]+[L(p-1)]\right)+p^{n-1}[L(p-2)].
\end{equation}
\end{theorem}

\begin{proof}
Since $A(1)\cong V(p-1)$ and $W(1)\cong L(p-2)$ as $W(1)$-modules, (\ref{w-1}) follows. Furthermore, note that the restricted baby Verma module $V(p-1)$ has two composition factors $L(0)$ and $L(p-1)$. This together with (\ref{w-1}) yields (\ref{w-2}).
\end{proof}

\section{Structure of the special algebras as modules over the Witt algebra}
In this section, we study the structure of the special algebra $S(n)$ as a module over the Witt algebra $W(1)$. For that, for each $ \un{a}=(a_1,\cdots,a_n)\in I^n, $  we define $$ x^{\un{a}}=x_1^{a_1}\cdots x_n^{a_n},\,\,\,\Omega(\un{a})=\{i\mid a_i\neq p-1\}\,\,\text{and}\,\,\ell(\un{a})=|\Omega(\un{a})|.$$  In the following, if $ \Omega(\un{a})=\{ i_1,\cdots,i_s \}, $ we always assume that $ 1\leq i_1<\cdots<i_s\leq n. $

\subsection{Basis of $ S(n) $} In this subsection, we give a basis of $S(n)$. This may be known for experts. However, we do not find it in literature.

Set
$$\mf{B}_1=\{ x^{\un{a}}\p_i\mid \un{a}\in I^n, 1\leq i\leq n,\text{ and } a_i=0, a_j\neq p-1 \text{ for some }j\neq i \}. $$
and
$$\mf{B}_2= \left\{ D_{i_ji_{j+1}}(x^{\un{a}}x_{i_j}x_{i_{j+1}}) \mid \un{a}\in I^n \text{ with } \Omega(\un{a})=\{i_1,\cdots,i_s\}, 1\leq j\leq s-1 \right\}.$$

Let $V_1$ be the linear subspace of $S(n)$ spanned by elements in $\mf{B}_1$, and let $V_2$ be the linear subspace of $S(n)$ spanned by $D_{ij}(x^{\un{a}})$ for $\un{a}\in I^n$ with $a_i\neq 0$ and  $a_j\neq 0$, $1\leq i<j\leq n$. Note that $ D_{ij}(x^{\un{b}})\in V_2 $ if and only if $ x^{\un{b}}=x^{\un{a}}x_ix_j $ for some $ \un{a}\in I^n $ such that $ a_i,a_j\neq p-1. $

We have the following basic observation.
\begin{lemma}\label{3.1}
Keep notations as above. Then the following statements hold.
\begin{itemize}
\item[(1)]	$ S(n)=V_1\ds V_2. $
\item[(2)] $\mf{B}_1$ is a basis of $V_1$.
\item[(3)] $\dim(V_1)=n(p^{n-1}-1)$, $\dim(V_2)=np^{n-1}(p-1)+1$.
\end{itemize}
\end{lemma}
\begin{proof}
(1) It's obvious that $ V_1\cap V_2=\{0\}. $ Moreover, if $\un{a}\in I^n$ with $a_i=0$ and $a_j\neq 0,$ then $ D_{ij}(x^{\un{a}})=a_jx^{\un{a}-\epsilon_j}\p_i\in V_1$. Hence, (1) follows.

(2) is obvious.

(3) The first assertion for $\dim(V_1)$ follows from (2). Furthermore, it follows from (1) and \cite[Theorem 3.7, Chapter 4]{SF} that
$$\dim(V_2)=\dim S(n)-\dim(V_1)= (n-1)(p^n-1)- n(p^{n-1}-1)=np^{n-1}(p-1)+1.$$
\end{proof}

The following result is crucial to our final determination of a basis of $S(n)$.

\begin{lemma}\label{v_2}
	Suppose $ \ell(\un{a})=s $ and $ \Omega(\un{a}) =\{i_1,\cdots,i_s \}. $
	Then for any $ k,l\in \Omega(\un{a}),$  $$ D_{kl}(x^{\un{a}}x_kx_l) \in \on{span}_{\bbf}\{D_{i_ji_{j+1}}(x^{\un{a}}x_{i_j}x_{i_{j+1}}) \mid 1\leq j\leq s-1\}. $$
\end{lemma}
\begin{proof}
Without loss of generality, we can suppose that $ k=i_1,\ l=i_s. $  It is readily shown that $ \{D_{i_ji_{j+1}}(x^{\un{a}}x_{i_j}x_{i_{j+1}}) \mid 1\leq j\leq s-1\} $ are linear independent. Recall that $ D_{kl}(x^{\un{a}}x_kx_l)= x^{\un{a}} ((a_l+1)x_k\p_k-(a_k+1)x_l\p_l). $ Since
	\[
	\left|
	\begin{matrix}
		a_{i_2}+1 & 0 & \cdots & 0 & a_{i_s}+1\\
		-(a_{i_1}+1) & \ddots & \ddots & \vdots & 0\\
		0 &\ddots & \ddots  & \ddots & \vdots\\
		0 & \cdots &  &  a_{i_s}+1 & 0 \\
		0&  \cdots &  & -(a_{i_{s-1}}+1) & -(a_{i_1}+1)
	\end{matrix}
	\right|=0,
	\]
	$D_{kl}(x^{\un{a}}x_kx_l) \in \text{span}_{\bbf}\{D_{i_ji_{j+1}}(x^{\un{a}}x_{i_j}x_{i_{j+1}}) \mid 1\leq j\leq s-1\}.$
\end{proof}

We need the following combination formulas for later use.
\begin{lemma}\label{equation}
The following equalities hold.
\begin{equation*}
	\sum_{s=0}^{n}sC_n^s(p-1)^s=n(p-1)p^{n-1}.
\end{equation*}
\begin{equation*}
\sum_{s=2}^{n}C_n^s(p-1)^{s-2}(s-1)=\sum_{i=1}^{n-1}ip^{i-1}.
\end{equation*}
\end{lemma}
\begin{proof}
Take derivations on both sides of the following equations respectively,
	$$ \sum_{s=0}^{n}C_n^s(x-1)^{s}
	=  x^n, $$
	$$ \sum_{s=2}^{n}C_n^s(x-1)^{s-1}
	=  \frac{1}{x-1}(x^n-n(x-1)-1)=\sum_{i=0}^{n-1}x^i-n,$$
	and put $ x=p. $ Then the desired equalities follows.
\end{proof}

As a consequence of Lemma \ref{v_2} and Lemma \ref{equation}, we have

\begin{corollary}\label{cor for S}
The subspace $ V_2 $ has a basis $\mf{B}_2$. Consequently, $S(n)$ has a basis $\mf{B}_1\cup\mf{B}_2$.
\end{corollary}
\begin{proof}
	By Lemma \ref{v_2}, $$ V_2= \sum_{D\in\mf{B}_2}\bbf D=\sum_{s=2}^n\sum_{\stackrel{\un{a}\in I^n}{\Omega(\un{a})=\{ i_1,\cdots,i_s\}}} \sum_{j=1}^{s-1} \bbf D_{i_ji_{j+1}}(x^{\un{a}}x_{i_j}x_{i_{j+1}}). $$
Moreover, it follows from Lemma \ref{3.1}(3) and Lemma \ref{equation} that
$$|\mf{B}_2|=\sum_{s=2}^{n}C_n^s(p-1)^{s}(s-1)=(p-1)^2 \sum_{s=2}^{n}C_n^s (p-1)^{s-2}(s-1)=(p-1)^2\sum_{i=1}^{n-1}ip^{i-1}=\dim(V_2). $$
Hence, $\mf{B}_2$ is a basis of $V_2$. This together with Lemma \ref{3.1}(1)(2) yields the second assertion.
\end{proof}

\subsection{The structure of $S(n)$ as a module over the Witt algebra}
Recall the Lie algebra embedding $\Theta_S: W(1)\inj S(n) $ given by $ \Theta_S(x_1^i\p_1)=D_{12}(x_1^ix_2)=x_1^i\p_1- ix_1^{i-1}x_2\p_2,\ i\in I. $ In this subsection, we study the structure of the special algebra $S(n)$ as a module over the Witt algebra $W(1)$. For that,  for each $ \un{l}=(l_2,\cdots l_n)\in I^{n-1},$ we denote $x^{\un{l}}=x_2^{l_2}\cdots x_n^{l_n}. $

For any $ 2\leq i \leq n$ and $ \un{l}\in I^{n-1}  $ with $ l_i=0, $  let $$ N_i:=\text{span}_{\bbf}\{ x_1^tx_2^{p-1}\cdots x_{i-1}^{p-1}x_{i+1}^{p-1}\cdots x_n^{p-1} \p_i \mid 0\leq t\leq p-2\}$$ and $ N_{\un{l},i}:=\text{span}_{\bbf}\{ x_1^tx^{\un{l}}\p_i\mid t\in I \} $ if $l_j\neq p-1 $ for some $ j\neq i $.
Then both $N_i$ and $N_{\un{l},i}$ are submodules over the Witt algebra. More precisely, we have
\begin{lemma}\label{5.7}
Keep notations as above. Then as modules over the Witt algebra $W(1)$, we have
\begin{enumerate}
		\item For each $ 2\leq i\leq n, $ $ N_i\cong L(p-1). $
		\item Suppose $ \un{l}\in I^{n-1}  $ with $ l_2=0 $ and $ l_j\neq p-1 $ for some $ j\neq 2.$ Then $ N_{\un{l},2}\cong V(0).  $
		\item For each $ 3\leq i \leq n, $ suppose $ \un{l}\in I^{n-1}  $ with $ l_i=0 $ and $ l_j\neq p-1 $ for some $ j\neq i.$ Then $ N_{\un{l},i}\cong V(p-1-l_2).  $
	\end{enumerate}
\end{lemma}
\begin{proof}
	(1) For each $ 2\leq i\leq n$, let $ \un{b}=(p-1,\cdots,p-1)-(p-1)\epsilon_i\in I^{n-1}. $ Then
	$$[\Theta_S(x_1^r\p_1),x_1^tx^{\un{b}} \p_i]=(t+r)x_1^{r+t-1}x^{\un{b}}\p_i,\,\,\forall\,r\in I,\,0\leq t\leq p-2. $$
	Therefore, $N_i$ is a simple $W(1)$-module with a maximal vector $ x_1^{p-2}x^{\un{b}} \p_i $ of weight $ p-1. $ Hence, $ N_i\cong L(p-1). $
	
	(2) Since $ \un{l}\in I^{n-1} $ with $ l_2=0, 	l_i\neq p-1 $ for some $ i\neq 2,
	$ we have
	$$ [\Theta_S(x_1^r\p_1),x_1^tx^{\un{l}} \p_2]= (t+r)x_1^{r+t-1}x^{\un{l}} \p_2,\,\,\forall\,r, t\in I. $$
	This implies that $ N_{\un{l},2}$, as a $W(1)$-module,  has a maximal vector $x_1^{p-1}x^{\un{l}} \p_2$ of weight $0$. Consequently,
$N_{\un{l},2}\cong V(0).  $

	(3)	For each $ i\geq 3, $ if $ \un{l}\in I^{n-1} $ with $ l_i=0,
	l_j\neq p-1 $ for some $ j\neq i, $ then
	$$[\Theta_S(x_1^r\p_1),x_1^tx^{\un{l}} \p_i ]=(t-rl_2)x_1^{r+t-1}x^{\un{l}} \p_i,\,\,\forall\,r, t\in I. $$
This implies that $ N_{\un{l},i}$  as a $W(1)$-module,  has a maximal vector $x_1^{p-1}x^{\un{l}} \p_i$ of weight $p-1-l_2$. Consequently, $ N_{\un{l},i}\cong V(p-1-l_2).  $ \qedhere
\end{proof}

For each $ 1\leq i<j\leq n\text{ and } \un{a}\in I^{n-1},  $ let $ M_{\un{a},ij}:=\text{span}_{\bbf}\{ D_{ij}(x_1^tx^{\un{a}})\mid t\in I \}.  $
\begin{lemma}\label{5.8}
	Suppose $ \un{l}=(l_2,\cdots,l_n)\in I^{n-1} $ with $\ell(\un{l})=s$ and $ \Omega(\un{l})=\{ i_1,\cdots , i_s \}\subseteq\{2,\cdots,n\}. $
Then as modules over the Witt algebra $W(1)$, we have
	\begin{enumerate}
		\item For each $ 1\leq j\leq s-1, $ $ M_{\un{l}+\epsilon_{i_j}+\epsilon_{i_{j+1}},i_ji_{j+1}}\cong
		V(p-1-l_2).  $
		\item If $ i_1=2, $ then $ M_{\un{l}+\epsilon_{2},12}\cong V(p-2-l_2). $
		\item If $ i_1> 2, $ then $ M_{\un{l}+\epsilon_{i_1},1i_1}\cong V(p-1). $
	\end{enumerate}
\end{lemma}
\begin{proof}
	(1) For each $ 1\leq j\leq s-1 $ and $ r,t\in I, $
	\[ [\Theta_S(x_1^r\p_1),D_{i_ji_{j+1}}(x_1^tx^{\un{l}}x_{i_j}x_{i_{j+1}})]=(t-rl_2)D_{i_ji_{j+1}}(x_1^{r+t-1}x^{\un{l}}x_{i_j}x_{i_{j+1}}).  \]
	This implies that $M_{\un{l}+\epsilon_{i_j}+\epsilon_{i_{j+1}},i_ji_{j+1}}$,  as a $W(1)$-module,  has a maximal vector
$D_{i_ji_{j+1}}(x_1^{p-1}x^{\un{l}}x_{i_j}x_{i_{j+1}})$ of weight $p-1-l_2$. Hence,  $ M_{\un{l}+\epsilon_{i_j}+\epsilon_{i_{j+1}},i_ji_{j+1}}\cong
		V(p-1-l_2).  $

	(2) For each $ r,t\in I, $ we have
	$$[\Theta_S(x_1^r\p_1),D_{12}(x_1^tx^{\un{l}+\epsilon_{2}})]=(t-rl_2-r)D_{ij}(x_1^{r+t-1}x^{\un{l}+\epsilon_{2}}). $$
This implies that $ M_{\un{l}+\epsilon_{2},12}$, as a $W(1)$-module,  has a maximal vector $D_{12}(x_1^{p-1}x^{\un{l}+\epsilon_{2}})$
of weight $p-2-l_2$. Hence, $ M_{\un{l}+\epsilon_{2},12}\cong V(p-2-l_2). $

	(3) Since $ i_1> 2, $ $ l_2=p-1. $ For each $ r,t\in I, $ we have
	$$ [\Theta_S(x_1^r\p_1),D_{1i_1}(x_1^tx^{\un{l}+\epsilon_{i_1}})] =tD_{1i_1}(x_1^{r+t-1}x^{\un{l}+\epsilon_{i_1}}). $$
This implies that  $ M_{\un{l}+\epsilon_{i_1},1i_1}$,  as a $W(1)$-module,  has a maximal vector $D_{12}(x_1^{p-1}x^{\un{l}+\epsilon_{i_1}})$ of weight $p-1$. Hence, $ M_{\un{l}+\epsilon_{i_1},1i_1}\cong V(p-1). $
\end{proof}

We are now in the position to present the following main result on the decomposition of $S(n)$ as a direct sum of $W(1)$-modules.

\begin{theorem}\label{thm for S}
As a $W(1)$-module, we have
$$\aligned S(n)&\cong\, V(0)^{\oplus\big((n-1)(p^{n-2}-1)+1\big)}\oplus V(p-1)^{\oplus\big((n-1)p^{n-2}-1\big)}\\
&\quad\,\oplus\Big(\bigoplus_{i=1}^{p-2}L(i)^{\oplus(n-1)p^{n-2}}\Big)\oplus L(p-1)^{\oplus(n-1)}.\endaligned$$
\end{theorem}

\begin{proof}
Set
$$ V_{11}:=\text{span}_{\bbf}\{ x^{\un{a}}\p_1\mid \un{a}\in I^n\text{ with } a_1=0 \,\text{and}\, a_j\neq p-1 \text{ for some }j\neq 1 \},$$
$$ V_{12}:=\text{span}_{\bbf}\{  x^{\un{a}}\p_i\mid 2\leq i\leq n, \,\un{a}\in I^n \text{ with } a_i=0\text{ and } a_j\neq p-1 \text{ for some }j\neq i \}. $$
Then  $ V_1=V_{11}\ds V_{12} $ as a vector space. Moreover, both $ V_{12} $ and $ V_{11}+V_2 $ are $W(1)$-modules, and $S(n)= V_{12}\oplus (V_{11}+V_2)$.
It follows from Lemma \ref{5.7} that $ V_{12}=U_1\ds U_2 \ds U_3 \ds U_4, $ where
	\[ U_1= \bds_{i=2}^n N_i\cong L(p-1)^{\oplus (n-1)}, \quad U_2= \bds_{\substack{\un{l}\in I^{n-1},\, l_2=0 \\ \un{l}\neq (0,p-1,\cdots,p-1)}} N_{\un{l},2}\cong V(0)^{\oplus (p^{n-2}-1)}, \]
	\[ U_3= \bds_{i=3}^n \bds_{j=0}^{p-2} \bds_{\substack{\un{l}\in I^{n-1}\\ l_i=0, \, l_2=j}} N_{\un{l},i}\cong \bds_{i=3}^n \bds_{j=0}^{p-2} \bds_{\substack{\un{l}\in I^{n-1}\\ l_i=0, \, l_2=j}} V(p-1-j) \cong \bigoplus_{i=1}^{p-1}V(i)^{\oplus (n-2)p^{n-3}},\]
	\[ U_4 = \bds_{i=3}^n \bds_{\substack{\un{l}\in I^{n-1}\\ l_i=0, \, l_2=p-1\\ \un{l}\neq \tau-(p-1)\epsilon_i}} N_{\un{l},i}\cong \bds_{i=3}^n \bds_{\substack{\un{l}\in I^{n-1}\\ l_i=0, \,l_2=p-1\\ \un{l}\neq \tau-(p-1)\epsilon_i}} V(0)\cong V(0)^{\oplus (n-2)(p^{n-3}-1)}. \]
Since $V(i)=L(i)$ for $1\leq i\leq p-1$ (see \S \ref{witt rep}), it follows that
\begin{eqnarray}\label{L_1}
V_{12}&\cong& V(0)^{\oplus((n-2)(p^{n-3}-1)+ p^{n-2}-1)}\oplus\Big(\bigoplus_{i=1}^{p-2}L(i)^{\oplus (n-2)p^{n-3}}\Big)\\ \cr
&&\oplus L(p-1)^{\oplus (n-1)}
\oplus V(p-1)^{\oplus (n-2)p^{n-3}}. \nonumber
\end{eqnarray}
It follows from Corollary \ref{cor for S} and Lemma \ref{5.8} that $ V_{11}+V_2=W_1\ds W_2\ds W_3 \ds W_4$,  where
	\[ W_1=\bds_{s=2}^{n-1} \bds_{\substack{\un{b}\in I^{n-1},\, i_1=2\\E(\un{b})=\{ i_1,\cdots,i_s \}}} \bds_{j=1}^{s-1} M_{\un{b}+\epsilon_{i_j} +\epsilon_{i_{j+1}},i_ji_{j+1}}, \quad W_2= \bds_{s=2}^{n-2} \bds_{\substack{\un{b}\in I^{n-1},\,i_1>2\\E(\un{b})=\{ i_1,\cdots,i_s \}}} \bds_{j=1}^{s-1} M_{\un{b}+\epsilon_{i_j} +\epsilon_{i_{j+1}},i_ji_{j+1}}, \]
	\[W_3= \bds_{s=1}^{n-1} \bds_{\substack{\un{b}\in I^{n-1},\,i_1=2\\E(\un{b})=\{ i_1,\cdots,i_s \}}} M_{\un{b}+\epsilon_{2} ,12}, \quad W_4= \bds_{s=1}^{n-2} \bds_{\substack{\un{b}\in I^{n-1},\,i_1>2\\E(\un{b})=\{ i_1,\cdots,i_s \}}} M_{\un{b}+\epsilon_{i_1} ,1i_1}. \]
Thanks to Lemma \ref{equation}, we have
$$\sum_{s=2}^{n-1} (s-1) C_{n-2}^{s-1} (p-1)^{s-1} = (n-2)(p-1)p^{n-3},$$
and
$$\sum_{s=2}^{n-2} (s-1)C_{n-2}^{s}(p-1)^{s}=(n-2)(p-1)p^{n-3}- (p^{n-2}-1).$$
Moreover, It follows from Lemma \ref{5.8} that
$$\displaystyle W_1\cong \bds_{i=1}^{p-1}\bds_{s=2}^{n-1} V(i)^{\oplus (s-1) C_{n-2}^{s-1} (p-1)^{s-1}}=\bds_{i=1}^{p-1} V(i)^{\oplus (n-2)(p-1)p^{n-3}}, $$
\[ W_2\cong \bds_{s=2}^{n-2} V(0)^{\oplus (s-1)C_{n-2}^{s}(p-1)^{s}}=V(0)^{\oplus ((n-2)(p-1)p^{n-3}- p^{n-2}+1)}, \]
\[ W_3\cong \bds_{i=0}^{p-2}V(i)^{\oplus(\bds_{s=1}^{n-1} C_{n-2}^{s-1}(p-1)^{s-1})}=\bds_{i=0}^{p-2}  V(i)^{\oplus p^{n-2}}, \]
\[ W_4\cong \bds_{s=1}^{n-2} V(p-1)^{\oplus C_{n-2}^{s}(p-1)^{s}}= V(p-1)^{\oplus (p^{n-2}-1)}. \]
Since $V(i)=L(i)$ for $1\leq i\leq p-1$ (see \S \ref{witt rep}), it follows that
\begin{eqnarray}\label{L_2}
V_{11}+V_2&\cong& V(0)^{\oplus((n-2)(p-1)p^{n-3}+1)}\oplus\Big(\bigoplus_{i=1}^{p-2}L(i)^{\oplus ((n-2)(p-1)p^{n-3}+p^{n-2})}\Big)\cr
&&\oplus V(p-1)^{\oplus ((n-2)(p-1)p^{n-3}+p^{n-2}-1)}.
\end{eqnarray}
Now the desired assertion follows from (\ref{L_1}) and (\ref{L_2}).
\end{proof}

As a  consequence of Theorem \ref{thm for S}, we further have

\begin{corollary}
As a module over the Witt algebra $W(1)$,
$$ [S(n)]=(2(n-1)p^{n-2}-n+1)[L(0)]+\sum_{i=1}^{p-2}(n-1)p^{n-2} [L(i)]+2(n-1)p^{n-2}[L(p-1)].$$
\end{corollary}
\begin{proof}
Since $[V(0)]=[V(p-1)]=[L(0)]+[L(p-1)]$, the assertion follows directly from Theorem \ref{thm for S}.
\end{proof}

\section{Structure of the Hamiltonian algebras as modules over the Witt algebra}
Recall the Lie algebra embedding $\Theta_H: W(1)\inj H(2r) $ given by $ \Theta_H(x_1^i\p_1)=-D_H(x_1^ix_{r+1})=x_1^i\partial_1-ix_1^{i-1}x_{r+1}\partial_{r+1}, i\in I. $ In this section, we study the structure of the  Hamiltonian algebra $H(2r)$ as a module over the Witt algebra $W(1)$.

We first investigate the case $H(2)$. Set
\[ H_j=
\begin{cases}
\text{span}_{\bbf}\{D_H(x_1^ix_2^j)\mid 1\leq i\leq p-1\}, &\on{ if } j=0,\cr
\text{span}_{\bbf}\{D_H(x_1^ix_2^j)\mid i\in I\}, & \on{ if } 1\leq j\leq p-2,\cr
\text{span}_{\bbf}\{D_H(x_1^ix_2^j)\mid 0\leq i\leq p-2\}, &\on{ if } j=p-1.
\end{cases}
\]
Since
\begin{equation}\label{h2-bracket}
[\Theta_H(x_1^s\p_1), D_H(x_1^ix_{2}^j)]= (i-sj)D_H(x_1^{s+i-1}x_{2}^j),\,\forall\, i,j,s\in I,
\end{equation}
each $H_j$ is a $W(1)$-module for any $j\in I$. Moreover, we have the following decomposition of $H(2)$ as a $W(1)$-module.
\begin{lemma}\label{h2}
As a $W(1)$-module, $H(2)$ is completely reducible and
$$H(2)\cong \Big(\bigoplus_{i=1}^{p-2}L(i)\Big)\oplus L(p-1)^{\oplus 2}.$$
In particular,
\begin{equation}\label{6.1}
	 [H(2)]=\sum_{j=1}^{p-2}[L(j)]+2[L(p-1)].
\end{equation}
\end{lemma}

\begin{proof}
For any $1\leq j\leq p-2$, it follows from (\ref{h2-bracket}) that $H_j$ has a maximal vector $x_1^{p-1}x_2^j$ of weight $p-1-j$. Since $H_j$ is $p$-dimensional, $H_j\cong V(p-1-j)\cong L(p-1-j)$ as  $W(1)$-modules. Furthermore, it again follows from (\ref{h2-bracket}) that $H_0$ has a maximal vector $x_1^{p-1}$ of weight $p-1$, and $H_{p-1}$ has a maximal vector $x_1^{p-2}x_2^{p-1}$ of weight $p-1$. And $\dim H_0=\dim H_{p-1}=p-1$, it follows that $H_0\cong H_{p-1}\cong L(p-1)$  as  $W(1)$-modules. Since $H(2)=\oplus_{j=0}^{p-1}H_j$, the desired assertion follows.
\end{proof}

\begin{remark}
(\ref{6.1}) was obtained in \cite[Lemma 2.6]{HS}.
\end{remark}

In general, for $r>1$, $j\in I$, $\un{l}=( l_2,\cdots, l_r, l_{r+2},\cdots, l_{2r})\in I^{2r-2},$ let $x^{\un{l}}:=x_2^{l_2}\cdots x_r^{l_r}x_{r+2}^{l_{r+2}}\cdots x_{2r}^{l_{2r}}$,
$$H_{j,\un{l}}:=\text{span}_{\bbf}\{ D_H(x_1^ix_{r+1}^jx^{\un{l}}) \mid i\in I \text{ such that } (i,j,\un{l})\neq (0,\cdots,0), (p-1,\cdots,p-1) \},$$
and
$$ H_{\un{l}}:=\text{span}_{\bbf}\{ D_H(x_1^ix_{r+1}^kx^{\un{l}}) \mid i,k\in I \text{ such that } (i,k,\un{l})\neq (0,\cdots,0), (p-1,\cdots,p-1) \}. $$
Since
\begin{equation}\label{h-bracket}
[\Theta_H(x_1^s\p_1), D_H(x_1^ix_{r+1}^jx^{\un{l}})]= (i-sj)D_H(x_1^{s+i-1}x_{r+1}^jx^{\un{l}}),\,\forall\, i,j,s\in I, \un{l}\in I^{n-2},
\end{equation}
both $H_{t,\un{l}}$ and $H_{\un{l}} $ are $W(1)$-modules for any $t\in I, \un{l}\in I^{2r-2}$. The following result describes the $W(1)$-module structure on $ H_{\un{l}} $.

\begin{lemma}\label{h-lem}
	Keep notations as above, then the following decompositions hold as $W(1)$-modules.
	\begin{enumerate}
		\item If $ \un{l}=(0,\cdots,0) $, then
$$H_{\un{l}}\cong
V(0)\oplus\Big(\bigoplus_{i=1}^{p-1} L(i)\Big).
$$
\item If  $ \un{l}=(p-1,\cdots,p-1), $ then
$$H_{\un{l}}\cong
\Big(\bigoplus_{i=1}^{p-1} L(i)\Big)\oplus V(p-1).
$$
\item  If $ \un{l}\neq(0,\cdots,0) $ and $ (p-1,\cdots,p-1), $ then
$$H_{\un{l}}\cong
V(0)\oplus\Big(\bigoplus_{i=1}^{p-2} L(i)\Big)\oplus V(p-1).
$$
	\end{enumerate}
\end{lemma}
\begin{proof}
Note that $H_{\un{l}}=\bigoplus_{j\in I}H_{j,\un{l}}$. We need to determine the $W(1)$-module structure of $H_{t,\un{l}}$ for any $t\in I$.
We just show the assertion for the case $ \un{l}=(0,\cdots,0) $. Similar arguments yield the assertion for the other cases.

In the following, we always assume that $ \un{l}=(0,\cdots,0) $. For each $0\leq j\leq p-1$, it follows from (\ref{h-bracket}) that $ H_{j,\un{l}}$, as a $W(1)$-module,  contains a maximal vector $D_H(x_1^{p-1}x_{r+1}^jx^{\un{l}})$ of weight $p-1-j$. Since $\dim H_{j,\un{l}}=p$ for $1\leq j\leq p-1$, $ H_{j,\un{l}}\cong V(p-1-j)$ as a $W(1)$-module. While $\dim H_{0,\un{l}}=p-1$, we have $H_{0,\un{l}}\cong L(p-1)$. Hence,
\[H_{\un{l}}=\bigoplus_{j\in I}H_{j,\un{l}}\cong L(p-1)\oplus\Big(\bigoplus_{j=1}^{p-1} V(p-1-j)\Big)\cong V(0)\oplus\Big(\bigoplus_{i=1}^{p-1} L(i)\Big).\qedhere\]
\end{proof}

We are now in the position to present the following main result on the structure of the Hamiltonian algebra $H(2r)$ as a module over the Witt algebra $W(1)$.
\begin{theorem}
As a module over the Witt algebra $W(1)$, we have
\begin{equation*}\label{h-decomp}
H(2r)\cong V(0)^{\oplus(p^{2r-2}-1)}\oplus \Big(\bigoplus_{i=1}^{p-2}L(i)^{\oplus p^{2r-2}}\Big)\oplus V(p-1)^{\oplus(p^{2r-2}-1)}\oplus L(p-1)^{\oplus 2}.
\end{equation*}	
In particular,
$$ [H(2r)]=(2p^{2r-2}-2) [L(0)]+p^{2r-2}\sum_{j=1}^{p-2}[L(j)]+2p^{2r-2}[L(p-1)]. $$
\end{theorem}

\begin{proof}
For $r=1$, the assertion follows from Lemma \ref{h2}. While for $r>1$, $H(2r)=\bigoplus_{\un{l}\in I^{2r-2}}H_{\un{l}}$, then the desired assertion follows directly from Lemma \ref{h-lem}.
\end{proof}

\section{Structure of the contact algebras as modules over the Witt algebra}
Recall the Lie algebra embedding $\Theta_K: W(1)\inj K(2r+1) $ given by $ \Theta_K(x_1^i\p_1)=\frac{1}{2}D_K(x_{2r+1}^i)=\sum_{j=1}^{2r}\frac{i}{2}x_{2r+1}^{i-1}x_j\partial_j+x_{2r+1}^{i}\partial_{2r+1}, i\in I. $ In this section, we study the structure of the  contact algebra $K(2r+1)$ as a module over the Witt algebra $W(1)$.

For $i\in I$, set $$\Gamma_i:=\left\{\un{l}=(l_1,\cdots l_{2r})\in I^{2r}\left|\, \frac{1}{2}\Big(\sum\limits_{j=1}^{2r} l_j-4\Big)\equiv i\,(\text{mod }p) \right. \right\}.$$
For arbitrary $i\in I$ and  $ (l_1,\cdots,l_{2r-1})\in I^{2r-1}, $  there is a unique $ l_{2r}\in I $ such that $$ l_{2r}\equiv 2i+4-\sum_{j=1}^{2r-1} l_j \,(\on{mod }p). $$ Hence,  $|\Gamma_i|=p^{2r-1}$ for any $i\in I$.

For any $ \un{l}=(l_1,\cdots l_{2r})\in I^{2r},$ we denote $x^{\un{l}}:=x_1^{l_1}\cdots x_{2r}^{l_{2r}} $ and
$$K_{\un{l}}:=\begin{cases}
\text{span}_{\bbf}\{x^{\un{l}}x_{2r+1}^t\mid 0\leq t\leq p-2\}, &\text{if } 2r+4\equiv 0\,(\text{mod }p)\text{ and } \un{l}=\tau-(p-1)\epsilon_{2r+1},\cr
\text{span}_{\bbf}\{x^{\un{l}}x_{2r+1}^t\mid 0\leq t\leq p-1\}, & \text{otherwise.}
\end{cases}
$$
Since
\begin{equation}\label{eq1 in K}
[\Theta_K(x_1^i\p_1),\, x^{\un{l}}x_{2r+1}^t]=\frac{1}{2}\Big(i\Big(\sum_{j=1}^{2r} l_j-2\Big)+2t\Big)x^{\un{l}}x_{2r+1}^{t+i-1},\,\forall\,i,t\in I,\, \un{l}\in I^{2r},
\end{equation}
$K_{\un{l}}$ is a $W(1)$-module for any $\un{l}\in I^{2r}$. More precisely, we have

\begin{lemma}\label{k lem}
Keep notations as above, then for any $ \un{l}=(l_1,\cdots l_{2r})\in I^{2r},$ as a $W(1)$-module
$$K_{\un{l}}\cong\begin{cases}
L(p-1), &\text{if } 2r+4\equiv 0\,(\text{\rm mod }p)\text{ and } \un{l}=\tau-(p-1)\epsilon_{2r+1},\cr
V\big(\frac{1}{2}\big(\sum_{j=1}^{2r} l_j\big)-2\big), & \text{otherwise.}
\end{cases}
$$
\end{lemma}

\begin{proof}
If $2r+4\equiv 0\,(\text{mod }p)\text{ and } \un{l}=\tau-(p-1)\epsilon_{2r+1}$, it follows from (\ref{eq1 in K}) that  $x^{\un{l}}x_{2r+1}^{p-2}$ is a maximal vector of weight $p-1$ in $K_{\un{l}}$, so that $K_{\un{l}}$, as a $W(1)$-module, is a homomorphic image of $V(p-1)$. Furthermore, since $\dim K_{\un{l}}=p-1$, it follows that $K_{\un{l}}\cong L(p-1)$.

If $2r+4\not\equiv 0\,(\text{mod }p)\text{ or } \un{l}=\tau-(p-1)\epsilon_{2r+1}$, it follows from (\ref{eq1 in K}) that  $x^{\un{l}}x_{2r+1}^{p-1}$ is a maximal vector of weight $\frac{1}{2}\big(\sum_{j=1}^{2r} l_j\big)-2$ in $K_{\un{l}}$. This together with the dimension of $K_{\un{l}}$ yields that  $K_{\un{l}}\cong V\big(\frac{1}{2}\big(\sum_{j=1}^{2r} l_j\big)-2\big).$
\end{proof}

As a consequence of Lemma \ref{k lem}, we have the following main result on the decomposition of the contact algebra $K(2r+1)$ as a module over the Witt algebra $W(1)$.

\begin{theorem}
As a $W(1)$-module, we have
$$K(2r+1)\cong\begin{cases}
V(0)^{\oplus (p^{2r-1}-1)}\oplus V(p-1)^{\oplus p^{2r-1}}\oplus\Big(\bigoplus\limits_{i=1}^{p-2} L(i)^{\oplus p^{2r-1}}\Big)\oplus L(p-1), &\text{if } 2r+4\equiv 0\,(\text{\rm mod }p),\cr
V(0)^{\oplus p^{2r-1}}\oplus V(p-1)^{\oplus p^{2r-1}}\oplus\Big(\bigoplus\limits_{i=1}^{p-2} L(i)^{\oplus p^{2r-1}}\Big), & \text{if } 2r+4\not\equiv 0\,(\text{\rm mod }p).
\end{cases}
$$
Consequently,
$$[K(2r+1)]=\begin{cases}
(2p^{2r-1}-1)[L(0)]+\sum\limits_{i=1}^{p-2}p^{2r-1}[L(i)]+2p^{2r-1}[L(p-1)], &\text{if } 2r+4\equiv 0\,(\text{\rm mod }p),\cr
2p^{2r-1}[L(0)]+\sum\limits_{i=1}^{p-2}p^{2r-1}[L(i)]+2p^{2r-1}[L(p-1)], & \text{if } 2r+4\not\equiv 0\,(\text{\rm mod }p).
\end{cases}
$$
\end{theorem}

\begin{proof}
Note that $K(2r+1)=\bigoplus_{\un{l}\in I^{2r}} K_{\un{l}}$,  It follows from Lemma \ref{k lem} that
$$K(2r+1)\cong\begin{cases}
V(0)^{\oplus (|\Gamma_0|-1)}\oplus V(p-1)^{\oplus |\Gamma_{p-1}|}\oplus\Big(\bigoplus\limits_{i=1}^{p-2} L(i)^{\oplus |\Gamma_{i}|}\Big)\oplus L(p-1), &\text{if } 2r+4\equiv 0\,(\text{\rm mod }p),\cr
V(0)^{\oplus |\Gamma_0|}\oplus V(p-1)^{\oplus |\Gamma_{p-1}|}\oplus\Big(\bigoplus\limits_{i=1}^{p-2} L(i)^{\oplus |\Gamma_{i}|}\Big), & \text{if } 2r+4\not\equiv 0\,(\text{\rm mod }p).
\end{cases}
$$
Since $|\Gamma_{i}|=p^{2r-1}$ for any $i\in I$, the first assertion holds. Moreover, since $$[V(0)]=[V(p-1)]=[L(0)]+[L(p-1)],$$ the second assertion follows.
\end{proof}

\textbf{Acknowledgment}
We would like to thank Prof. A. Premet for helpful discussion on the basis of $S(n)$.

\end{document}